\documentclass[11pt]{article}
\pdfoutput=1
\usepackage{amsmath, amssymb, amscd, amsthm, amsfonts}
\usepackage{indentfirst}

\usepackage{graphicx}

\usepackage{hyperref}
\numberwithin{equation}{section}

\hypersetup{colorlinks=true,linkcolor=black}

\title{Anderson localization for  block jacobi operators with quasi-periodic meromorphic potential }

\author{Xiaojian Zhang}

\date{}

\newtheorem{theorem}{Theorem}

\newtheorem{lemma}[theorem]{Lemma}

\newtheorem{proposition}[theorem]{Proposition}

\begin{document}

	\maketitle

	\begin{abstract}
		In this paper, we study block Jacobi operators on $\mathbb{Z}$ with quasi-periodic meromorphic potential. We prove the non-perturbative Anderson localization for such operators in the large coupling regime.
	\end{abstract}

% The case $r=2$ was proved first by Radon, as a lemma in his proof of Helly's theorem.

\section{Introduction and main result }\label{section-introduction}

The study of spectral properties  of  discrete quasi-periodic Schr\"odinger operators has attracted great attention over the years. Of particular importance is the phenomenon of Anderson localization, which means the operator has  pure point spectrum with exponentially decaying eigenfunctions.  

The Anderson localization for almost Mathieu operator (AMO) in the large potential regime was first proved in the fundamental works of  Sinai \cite{ref15} and Fr\"ohlich-Spencer-Wittwer \cite{ref16}. These results rely on KAM type perturbative method. Along this line, Eliasson \cite{ref17} proved via the KAM reducibility method  pure point spectrum for a class of quasi-periodic Schr\"odinger operators with certain Gevrey regular potential. Later, Jitomirskaya  made the great breakthrough in \cite{ref18, ref19}, where a non-perturbative method was first developed for the AMO. 

In 2000, Bourgain-Goldstein \cite{ref9} extended the non-perturbative method of Green's function estimates to general analytic quasi-periodic Schr\"odinger operators 
\begin{equation}\label{fc1}	H_{\lambda}(x)=\lambda v(x+n\omega)+\Delta
\end{equation} defined on $ \mathcal{l}^{2}(\mathbb{Z})$,
where 
\begin{equation}\label{key}
	\Delta(n, n^{\prime})=\left\{
	\begin{aligned}
		1,\quad |n-n^{\prime}|=1,\\
		0,\quad |n- n^{\prime}|\neq 1,
	\end{aligned}\right.
	\quad  n, n^{\prime}\in\mathbb{Z}.
\end{equation} In this remarkable work, the authors introduced the new ingredients of large deviation estimates concerning the Lyapunov exponent together with the semi-algebraic sets theory.  Based on the method of \cite{ref4}, Bourgain-Jitomirskaya  \cite{ref10} generalized the results of \cite{ref9} to the band operators case
\begin{equation}\label{uii}
	H_{(n,s),(n^{\prime},s^{\prime})}(\omega,\theta)=\lambda \delta_{nn^{\prime}}\delta_{ss^{\prime}}v_{s}(\theta+n\omega)+\Delta    \quad (n\in\mathbb{Z}  ,s=1,...,b),
\end{equation} where 
\begin{equation}\label{key}
	\Delta((n,s),(n^{\prime},s^{\prime}))=\left\{
	\begin{aligned}
		&1,   \quad |n-n^{\prime}|+|s-s^{\prime}|=1\\
		&0,    \quad   
		\rm{otherwise} 
	\end{aligned}\right.
\end{equation}
and $ v_{s}(s=1,...b) $  non-constant  analytic on $ \mathbb{T} $.

In \cite{ref2}, Klein developed the results of \cite{ref10} further and established non-perturbative Anderson localization for block Jacobi operators 
\begin{equation}\label{uvv}
	[H_{\lambda}(x)\vec{\varphi}]_{n}:=-(W_{n+1}(x)\vec{\varphi}_{n+1}+W^{T}_{n}(x)\vec{\varphi}_{n-1}+R_{n}(x)\vec{\varphi}_{n})+\lambda F_{n}(x)\vec{\varphi}_{n} \quad n\in \mathbb{Z},
\end{equation}	
where $W_{n}(x),R_{n}(x)  $ and $ F_{n}(x) $ are   quasi-periodic functions defined by  \begin{equation}\label{key}
	W_{n}(x):=W(x+n\omega),R_{n}(x):=R(x+n\omega),F_{n}(x):=F(x+n\omega),
\end{equation}
with $ W(x),R(x),F(x) $ being   $ l $-order symmetric matrix-valued functions  and  all the elements are analytic functions defined on 	 $\mathbb{T}=\mathbb{R}/2\pi \mathbb{Z}$. Very recently, Jian-Shi-Yuan \cite{ref12} extended the results of \cite{ref2} to the long-range block operators case
\begin{equation}\label{uuvv}
	(H_{\epsilon,\omega}(x)\vec{\varphi})_{n}:=\epsilon \sum_{k\in \mathbb{Z}}^{}W_{k}\vec{\varphi}_{n-k}+V(x+n \omega)\vec{\varphi}_{n},
\end{equation}
where $ V(x)=diag\{v_{1}(x),...,v_{l}(x)\} $. $ v_{i}(x)(1\leq i \leq l) $ are non-constant real analytic functions on $ \mathbb{T} $. $ W_{k} (k\in \mathbb{Z}) $ are $ l\times l $ matrices satisfying $ W_{k}^{\star}=W_{-k}(W_{k}^{\star} $ is the conjugate transpose of $ W_{k} $), $ \vert\vert W_{k}^{\star} \vert\vert\leq e^{-\rho |k|} $  and $  \rho>0 $. We should remark that the above mentioned works require the potential to be analytic.  In fact, it is an important problem wether or not the non-perturbative Anderson localization still holds for operators with non-analytic quasi-periodic potentials. Our main motivation comes from this problem. 

In this paper, 
we study the Anderson localization of (\ref{uvv}) where the \textbf{ elements of diagonal} of $ V_{n}(x) , (n\in \mathbb{Z})$ are \textbf{meromorphic functions}, where $ V_{n}(x)=\lambda F_{n}(x)+R_{n}(x),  (n\in \mathbb{Z})$. 

Our main result is 
\begin{theorem}
{Given a positive integer $ l $. Let  $ W(x)=(W_{ij}(x):1\leq i\leq l,1\leq j\leq l), R(x)=(R_{ij}(x):1\leq i\leq l,1\leq j\leq l) $ and $ F(x)=(F_{ij}(x):1\leq i\leq l,1\leq j\leq l) $ be $ l\times l $ symmetric matrices for $\forall x\in \mathbb{T}$. Suppose that  
		
		(i). each entry of $ W(x), R(x)-{ diag}\{R_{ii}(x), 1\leq i\leq l\} $ and  $ F(x)-{ diag}\{F_{ii}(x), 1\leq i\leq l\} $ is real analytic in $ x\in \mathbb{T} $.
		
		(ii). For all $ 1\leq i\leq l $, the diagonal elements $ F_{ii}(x)$  and  $ R_{ii}(x)  $ are meromorphic in $ x\in \mathbb{T} $. They can be written as 
		\begin{equation}\label{a8}
			F_{ii}(x)=\dfrac{\tilde{\phi}^{F}_{ii}(x)}{\phi^{F}_{ii}(x)},R_{ii}(x)=\dfrac{\tilde{\phi}^{R}_{ii}(x)}{\phi^{R}_{ii}(x)},
		\end{equation} where $ \tilde{\phi}^{F}_{ii}(x),\phi^{F}_{ii}(x), \tilde{\phi}^{R}_{ii}(x)  $ and $ \phi^{R}_{ii}(x)  $ are real analytic in $ \mathbb{T} $, and   the number of zeros of $ \phi^{F}_{ii}(x)  $ and  $ \phi^{R}_{ii}(x)  $ is finite for all $ 1\leq i \leq l $.  
		
		(iii). For $\forall t\in \mathbb{R}$,
		\begin{equation}\label{a6}
			{	det}[(F(x)-tI)M(x)]\not\equiv 0 
		\end{equation}as a function of $ x\in\mathbb{T} $, where $ I $ is $ l\times l $ unit matrix, and  \begin{equation}\label{a7}
			M(x)={ diag}\{\phi^{F}_{ii}(x),1\leq i \leq l \}{ diag}\{\phi^{R}_{ii}(x),1\leq i \leq l \}. 
		\end{equation} 
		
		Let \begin{equation}\label{key}
			W_{n}(x):=W(x+n\omega),R_{n}(x):=R(x+n\omega),F_{n}(x):=F(x+n\omega), n\in \mathbb{Z}. 
		\end{equation}
		Assume that  $\omega \in \mathbb{T}$ satisfies Diophatine condition($ { DC}_{A,C_{0}} $):
		\begin{equation}\label{key}
			\vert\vert	k\omega\vert\vert \ge C_{0}\dfrac{1}{|k|^{A}}	
		\end{equation}
		for all  $ k  \in \mathbb{Z}\backslash\{0\} $ , where  $  C_{0}> 0 $, $ A> 1 $ are constants, and  $ \vert\vert	k\omega\vert\vert=\mathop{min}\limits_{j\in\mathbb{Z}}|k\omega-j| $.
		
		Then there exists   $ \lambda_{0}=\lambda_{0}(l,F,R,W)>0 $  such that if $\lambda>\lambda_{0}$, the  block Jacobi operator 
		\begin{equation}\label{a9}\begin{aligned}
				[H_{\lambda}(x)\vec{\varphi}]_{n}:=-(W_{n+1}(x)\vec{\varphi}_{n+1}+W^{T}_{n}(x)\vec{\varphi}_{n-1}+R_{n}(x)\vec{\varphi}_{n})+\lambda F_{n}(x)\vec{\varphi}_{n}, \\ \lambda>0, x\in \mathbb{T},  
				\forall \vec{\varphi}=(\vec{\varphi}_{n})_{n\in \mathbb{Z}}\in \mathcal{l}^{2}(\mathbb{Z},\mathbb{R}^{l})\end{aligned} 
		\end{equation} satisfies Anderson localization for fixed $ x_{0}\in \mathbb{T} $ and almost every $\omega\in DC_{A,C_{0}} $.}
\end{theorem}	
Remark: we remark that
 
(i). We extend the result of Klein\cite{ref2}. In his model, the entrices of matrices $ W(x) $, $ R(x) $ and  $ F(x) $ are  all analytic functions. Both of our results are nonperturbative. 

(ii). The operator that \cite{ref2} studied  is bounded self-adjoint.
In our paper,  the operator contains meromorphic terms. Then the operator we study is not bounded. 
This leads to some difficulities. For example, the associated energy $ E $ of a bounded self-adjoint operator can be restricted to one  compact interval of $ \mathbb{R} $. 
In our case, the energy $ E $ could not more be restricted to any compact interval. With some  tricks, we solve the  problem of singularity brought by meromorphic fucntions.

%	\qquad Consider the  quasi-periodic Schr$\ddot{o}$dinger operator on $\{\varphi_{n}\} \in l^{2}(Z)$ defined by
%	
%	\begin{equation}		
	%		[H_{\lambda}(x)\varphi]_{n}:=-(\varphi_{n+1}+\varphi_{n-1}-2\varphi_{n})+\lambda f(x+n\omega)\varphi_{n}
	%	\end{equation}
%	
%	where $\lambda\gg0$,indexed by $x\in \mathbb{T}=\mathbb{R}/\mathbb{Z}$, $f$ is analytic ,$\omega$  fixed irrational satisfy DC .
%	
%	In this paper we study the  block jacobi operator model  which is  one type of matrix-valued Schrodinger-type operator.
%	
%	The corresbounding $H_{\lambda}(x)$ define on 	
%	$\{\vec{\varphi}_{n}\}\in l^{2}(Z,\mathbb{R}^{d})$ of the form 	
%	
%	
%	
%	\begin{equation}\label{key}
	%		[H_{\lambda}(x)\vec{\varphi}]_{n}:=-(W_{n+1}(x)\vec{\varphi}_{n+1}+W^{T}_{n}(x)\vec{\varphi}_{n-1}+R_{n}(x)\vec{\varphi}_{n})+\lambda F_{n}(x)\vec{\varphi}_{n} 
	%	\end{equation}
%	
%	$W_{n}(x),R_{n}(x),F_{n}(x)$ are  generated by  matrix-value function 	$W(x),R(x),F(x):T\rightarrow Mat_{d}(\mathbb{R}),Mat_{d}(\mathbb{R})$ denote the $d\times d$ matrix .

%	The Anderson localization of $H_{\lambda}(x)$ when $W(x),R(x),F(x)$ are analytic have been proved.

\section{Upper bound of the minor}

For all integer $ N>1$, define \begin{equation}\label{a0}
	H_{N}(x)=R_{[1,N]}H_{\lambda}(x)R_{[1,N]},
\end{equation}
where $ R_{[1,N]} =$ coordinate restriction to $[1,N]\subset\mathbb{Z}$, and the associated Green's function $ G_{N}(x,E)=[H_{N}(x)-E)]^{-1} $.

Let \begin{equation}\label{a1}
	\tilde{H}_{N}(x,E):=[H_{N}(x)-E]{ diag}\{\dfrac{1}{\sqrt{1+E^{2}}}M_{j}(x), 1\leq j\leq l \},
\end{equation}
where \begin{equation}\label{a2}
	M_{j}(x)={ diag}\{\phi^{F}_{ii}(x+j\omega),1\leq i\leq l \}{ diag}\{\phi^{R}_{ii}(x+j\omega),1\leq i\leq l \}=M(x+j\omega).
\end{equation}
Then 
\begin{equation}\label{a3}
	G_{N}(x,E)={ diag}\{\dfrac{1}{\sqrt{1+E^{2}}}M_{j}(x), 1\leq j\leq l \}\tilde{H}^{-1}_{N}(x,E).
\end{equation}

Note that $ H_{N}(x) $ and  $	G_{N}(x,E) $ are $Nl\times Nl $ matrices. For 
$1\leq \alpha \leq Nl ,1\leq \alpha^{\prime}\leq Nl $, by $ \tilde{H}^{-1}_{N}(x,E)(\alpha ,\alpha^{\prime}) $ and $ G_{N}(x,E)(\alpha ,\alpha^{\prime}) $ denote the 	$ (\alpha ,\alpha^{\prime}) $-entry of $ \tilde{H}^{-1}_{N}(x,E) $ and $ G_{N}(x,E) $ respectively.
Then by Cramer's rule,  \begin{equation}\label{a4}
	|\tilde{H}^{-1}_{N}(x,E)(\alpha ,\alpha^{\prime})|=\frac{|\tilde{\mu}_{N}(x,E)_{(\alpha,\alpha^{\prime})}|}{|{ det}[\tilde{H}_{N}(x,E)]|},	\end{equation}
where 
$\tilde{\mu}_{N}(x,E)_{(\alpha,\alpha^{\prime})} $ is 
$ (\alpha^{\prime},\alpha) $-minor of $ \tilde{H}_{N}(x,E) $.
Noting that $H_{\lambda}(x)=((H_{\lambda}(x))_{i,j})_{i,j\in \mathbb{Z}}$, where 
$(H_{\lambda}(x))_{i,j}$ is   
$l\times l$ matrix . We have that there exists integers $ p(\alpha),p(\alpha^{\prime}) \in [0,N-1] $, $ q(\alpha),q(\alpha^{\prime}) \in [1,l] $, such that 
\begin{equation}\label{a5}
	\alpha=p(\alpha)l+q(\alpha),\alpha^{\prime}=p(\alpha^{\prime})l+q(\alpha^{\prime}).
\end{equation}
Furthermore, by $ \eqref{a1}, \eqref{a2}, \eqref{a3}, \eqref{a4} $ and $ \eqref{a5} $,  \begin{equation}\label{e4}
	|G_{N}(x,E)(\alpha,\alpha^{\prime})|=\frac{1}{\sqrt{1+E^{2}}}|\phi^{F}_{qq}(x+(p+1)\omega)\phi^{R}_{qq}(x+(p+1)\omega)\frac{\tilde{\mu}_{N}(x,E)_{(\alpha,\alpha^{\prime})}}{det[\tilde{H}_{N}(x,E)]}|,
\end{equation}
where $ q=q(\alpha),p=p(\alpha^{\prime}) $.
\begin{proposition}
	
	{There exists a constant } $  C=C(W,F,R,l)>0$ such that   
	\begin{equation}\label{e2}
		\frac{1}{Nl}{ log}|\tilde{\mu}_{N}(x,E)_{(\alpha,\alpha^{\prime})}|\leq -\frac{|p(\alpha)-p(\alpha^{\prime})|}{Nl}{ log}(\lambda+|E|)+{ log}(1+\frac{\lambda}{|E|})+C
	\end{equation}
	for 
	$1\leq \alpha \leq Nl ,1\leq \alpha^{\prime}\leq Nl $ and  $x\in \mathbb{T}$.
\end{proposition}  	

\begin{proof}
	For 
	$1\leq i \leq N ,1\leq j\leq N $, by $ \tilde{H}_{N}(x,E)(i ,j) $ denote the 	$ (i ,j) $-block of  $ N\times N $ block matrix $ \tilde{H}_{N}(x,E) $. By \eqref{a0}, \eqref{a1} and \eqref{a2}, 
	\begin{equation}\label{c11}		\tilde{H}_{N}(x,E)(i,j)=\left\{
		\begin{aligned}
			\frac{1}{\sqrt{1+E^{2}}}\tilde{V}_{j}(x),&&  i=j,  \\
			\frac{1}{\sqrt{1+E^{2}}}(-\tilde{W}_{i}^{T}(x)),&&  i=j+1,  \\
			\frac{1}{\sqrt{1+E^{2}}}(-\tilde{W}_{j}(x)),&& i=j-1, \\
			0,&& { otherwise},
		\end{aligned}\right.
	\end{equation}
	where
	\begin{equation}\label{b1}
		\tilde{V}_{j}(x)=[\lambda F_{j}(x)+R_{j}(x)-EI]M_{j}(x),
	\end{equation}  
	\begin{equation}\label{b2}
		-\tilde{W}_{j}(x)=(-W_{j}(x))M_{j}  \quad 
	\end{equation} 	
	and \begin{equation}\label{b3}
		-\tilde{W}_{j}^{T}(x)=(-W^{T}_{j}(x))M_{j-1}(x).
	\end{equation}	
	
	In view of \eqref{b1}, \eqref{b2} and \eqref{b3},  there exists a constant   $ C=C(W,R,F)> 0 $ such that for  $ 1\leq  j\leq N $,
	\begin{equation}\label{dd1}
		\vert\vert \tilde{W}_{j}(x)\vert\vert_{\infty} \leq C 
	\end{equation}
	and 
	
	\begin{equation}\label{dd2}
		\vert\vert \tilde{V}_{j}(x)\vert\vert_{\infty} \leq C(\lambda+|E|).
	\end{equation}
	
	Let $ g:=\sqrt{1+E^{2}}\tilde{H}_{N}(x,E) $.  By $ 	\mathbb{P}_{\alpha^{\prime}\to \alpha}$ denote the set of all paths  \begin{equation}\label{key}
		\{\gamma=(r_{1},r_{2},...,r_{s}) :( 1\leq s \leq Nl)\},
	\end{equation}
	where 
	
	(i) for $ i=1,...,s  ,$\begin{equation}\label{key}
		r_{i}\in [1,Nl]\subset\mathbb{Z}.
	\end{equation} 
	
	(ii) For $\forall 1< i < s , 1< j < s$ and $ i\neq j $, 
	\begin{equation}\label{key}
		r_{i}\neq r_{j}, r_{i}\neq r_{1},r_{i}\neq r_{s}.
	\end{equation} 
	
	(iii) For $ i=1, j=s $,
	\begin{equation}\label{key}
		r_i=\alpha^{\prime},r_j=\alpha.
	\end{equation}
	By $ g(r_{i},r_j) $ denote the 
	$ (r_{i},r_j) $-entry of $ Nl\times Nl $ matrix $ g $. By  $g_{\urcorner\alpha,\urcorner\alpha^{\prime}} $ denote  the matrix remaining after excluding the row $ \alpha $ and column $ \alpha^{\prime}  $ of $ g $. By  $g_{\urcorner\gamma,\urcorner\gamma^{\prime}} $ denote  the matrix remaining after excluding the row $ r_{1},r_{2},...,r_{s} $ and column $ r_{1},r_{2},...,r_{s}  $ of $ g $. Then by expansion of a determinant by  row and column,
	
	\begin{equation}\label{h2}
		\tilde{\mu}_{N}(x,E)_{(\alpha,\alpha^{\prime})}=(\frac{1}{\sqrt{1+E^{2}}})^{Nl}det[g_{\urcorner\alpha,\urcorner\alpha^{\prime}}]=(\frac{1}{\sqrt{1+E^{2}}})^{Nl}(-1)^{\alpha+\alpha^{\prime}}\sum_{\gamma\in \mathbb{P}_{\alpha^{\prime}\to \alpha}}^{}(-1)^{|\gamma|+1}c(\gamma)det[g_{\urcorner\gamma,\urcorner\gamma^{\prime}}],
	\end{equation}
	where $c(\gamma)=g(r_{1},r_{2})\dots g(r_{s-1},r_{s})$ and $ |\gamma|=s $ is the lenth of the path  $ \gamma $.

	Using  Hadamard's  inequality, we have 
	\begin{equation}\label{h3}
		|{ det}[g_{\urcorner\gamma,\urcorner\gamma^{\prime}}]|\leq C(\lambda+|E|)^{Nl-|\gamma|}.
	\end{equation}
	Then by \eqref{h2} and  \eqref{h3},
	\begin{equation}\label{d1}
		|{ det}[g_{\urcorner\alpha,\urcorner\alpha^{\prime}}]|\leq \sum_{\gamma\in \mathbb{P}_{\alpha^{\prime}\to \alpha}}|c(\gamma)|(C(\lambda+|E|))^{Nl-|\gamma|}.
	\end{equation}	
	
	We are now in position to investigate 	$ c(\gamma) $. Note that 
	$ c(\gamma) $ depends on the path $ \gamma=(r_{1},r_{2},...,r_{s})	 $, where  $ 1\leq r_{i} \leq Nl,(i=1,...s). $ Write 
	$ r_{i}=p(r_{i})l+q(r_{i}),0\leq 	p(r_{i})\leq N-1, 1\leq q(r_{i})\leq l.  $
	
	For $ 1\leq j \leq N, 1\leq m\leq l $ and $ 1\leq n \leq l $,  by  $\tilde{V}_{j}(x)(m,n),-\tilde{W}_{j}(x)(m,n) $ and $ -\tilde{W}_{j}^{T}(x)(m,n) $ denote the 
	$ (m,n) $-entry of $ l\times l $ matrices $ \tilde{V}_{j}(x),-\tilde{W}_{j}(x) $  and $ -\tilde{W}_{j}^{T}(x)$ respectively. By   \eqref{c1}, for $i,j=1,...s  $,
	
	\begin{equation}\label{e998}
		g(r_{i},r_{j})=\left\{
		\begin{aligned}
			\tilde{V}_{p(r_{i})+1}(x)(q(r_{i}),q(r_{j})),&&   p(r_{i})=p(r_{j}) , \\
			-\tilde{W}^{T}_{p(r_{i}+1)}(x)(q(r_{i}),q(r_{j})),&&   p(r_{i})-p(r_{j})=1,  \\
			-\tilde{W}_{p(r_{i})+1}(x)(q(r_{i}),q(r_{j})),&&   p(r_{j})-p(r_{i}) =1,\\
			0,&& { otherwise}.
		\end{aligned}\right.
	\end{equation}
	
	Then  by \eqref{dd1} ,\eqref{dd2} and \eqref{e998}, we have
	
	\begin{equation}\label{k1}
		|g(r_{i},r_{j})|\leq \left\{
		\begin{aligned}
			C(\lambda+|E|),&&   p(r_{i})=p(r_{j}),  \\
			C ,&&|p(r_{j})-p(r_{i})| =1,  \\		 0,&& { otherwise}.
		\end{aligned}\right.
	\end{equation}
Consider the  set	
	
	\begin{equation}
		\mathbb{P}_{\alpha^{\prime}\to \alpha}^{\star}:=\{\gamma=(r_{1},r_{2},...,r_{s})\in \mathbb{P}_{\alpha^{\prime}\to \alpha},|g(r_{i},r_{j})|\neq 0 ,\forall 1\leq i\leq s-1\}.
	\end{equation}
	By $ b(\gamma) $ denote the number of elements of set 
	\begin{equation}\label{e1}
		\{1\leq i\leq s-1:\gamma=(r_{1},r_{2},...,r_{s})\in \mathbb{P}_{\alpha^{\prime}\to \alpha},|p(r_{i})-p(r_{i+1})| =1 \}.
	\end{equation} 	
	
	Hence by \eqref{k1} and \eqref{e1},
	\begin{equation}\label{d2}
		|c(\gamma)|=\prod_{i=1}^{s-1}|g_{r_{i},r_{i+1}}|\leq C^{b(\gamma)}(C(\lambda+|E|))^{s-b(\gamma)}=C^{|\gamma|}(\lambda+|E|)^{|\gamma|-b(\gamma)}. 
	\end{equation}
	By \eqref{d1} and \eqref{d2},
	
	\begin{equation}\label{key99}	
		|{ det}[g_{\urcorner\alpha,\urcorner\alpha^{\prime}}]|\leq \sum_{\gamma\in \mathbb{P}_{\alpha^{\prime}\to \alpha}}C^{|\gamma|}(\lambda+|E|)^{|\gamma|-b(\gamma)} (C(\lambda+|E|))^{Nl-|\gamma|}. \end{equation}
	Furthermore by \eqref{e1},	
	
	\begin{equation}\label{f1}
		|p(\alpha)-p(\alpha^{\prime})|\leq \sum_{j=1}^{s-1}|p(r_{j+1})-p(r_{j})|=b(\gamma)\leq Nl.
	\end{equation}
	
	Then by \eqref{d1}, \eqref{d2} and \eqref{f1}, 
	\begin{equation}\label{g1}
		|{ det}[g_{\urcorner\alpha,\urcorner\alpha^{\prime}}]|\leq\sum_{b=|p(\alpha)-p(\alpha^{\prime})|}^{Nl}\sum_{\gamma\in \mathbb{P}^{*}_{\alpha^{\prime}\to \alpha},b(\gamma)=b}^{}C^{|\gamma|}(\lambda+|E|)^{|\gamma|-b(\gamma)} (C(\lambda+|E|))^{Nl-|\gamma|}.
	\end{equation} 
	
	Notice that the number of $ \gamma\in \mathbb{P}^{*}_{\alpha^{\prime}\to \alpha} $ is at most$  (3l)^{Nl} $. Then  by \eqref{g1}, 
	\begin{equation}
		|{ det}[g_{\urcorner\alpha,\urcorner\alpha^{\prime}}]|\leq \sum_{b=|p(\alpha)-p(\alpha^{\prime})|}^{Nl} (3lC)^{Nl} (\lambda+|E|)^{Nl-b}=(3lC)^{Nl} (\lambda+|E|)^{Nl}\sum_{b=|p(\alpha)-p(\alpha^{\prime})|}^{Nl}(\lambda+|E|)^{-b}.
	\end{equation}
	Thus 
	\begin{equation}\label{h1}
		|{ det}[g_{\urcorner\alpha,\urcorner\alpha^{\prime}}]|\leq(3lC)^{Nl} (\lambda+|E|)^{Nl}(\lambda+|E|)^{-|p(\alpha)-p(\alpha^{\prime})|}.
	\end{equation}
	By \eqref{h2}, \eqref{h1},

	\begin{equation}
		\tilde{\mu}_{N}(x,E)_{(\alpha,\alpha^{\prime})}\leq(3lC)^{Nl} (\frac{\lambda}{|E|}+1)^{Nl}(\lambda+|E|)^{-|p(\alpha)-p(\alpha^{\prime})|}.
	\end{equation}
	Taking the logarithm on both sides,  
	\begin{equation}
		\frac{1}{Nl}{ log}|\tilde{\mu}_{N}(x,E)_{(\alpha,\alpha^{\prime})}|\leq -\frac{|p(\alpha)-p(\alpha^{\prime})|}{Nl}{ log}(\lambda+|E|)+{ log}(1+\frac{\lambda}{|E|})+C.
	\end{equation}
\end{proof}	
\section{Lower bound of the determinant}
\begin{proposition}
	{ Assume that  \eqref{a6} and \eqref{a7} hold, then there exists constants $\lambda_{1}=\lambda_{1}(W,F,R)>0$ and $C_{1}=C_{1}(W,F,R)>0$ such that if} $\lambda > \lambda_{1}$, 
		\begin{equation}\label{key}
		\int_{\mathbb{T}} \dfrac{1}{Nl}\log |det(\tilde{H}_{N}(x,E))|dx\geq \log\lambda-C_{1}
	\end{equation}
	{ for all $ \omega\in \mathbb{T}$, $N\in\mathbb{N} $ and} $ x\in \mathbb{T} .$
	
\end{proposition}
	\begin{proof} 
For $1\leq \alpha \leq Nl ,1\leq \alpha^{\prime}\leq Nl $, by $ \tilde{H}(x,E)(\alpha ,\alpha^{\prime}) $ denote the 	$ (\alpha ,\alpha^{\prime}) $-entry of $ Nl\times Nl $ matrix $ \tilde{H}_{N}(x,E) $. 

By \eqref{a9},\eqref{a0},\eqref{a1} and \eqref{a2}, for given $ \rho>0 $, denote by $ \tilde{H}_{N}(z,E)(\alpha ,\alpha^{\prime}) ,
( 1\leq \alpha \leq Nl ,1\leq \alpha^{\prime}\leq Nl) $ the analytic continuation of  $ \tilde{H}_{N}(x,E)(\alpha ,\alpha^{\prime}),
(1\leq \alpha \leq Nl ,1\leq \alpha^{\prime}\leq Nl)$ to the  strip  \begin{equation}\label{key}
	\mathcal{A}_{\rho}=\{z:z=x+\sqrt{-1}y, x\in [0,1], |y|\leq \rho \}.  
\end{equation}
 Let $ \tilde{H}_{N}(z,E):=(\tilde{H}_{N}(z,E)(\alpha ,\alpha^{\prime}))_{(1\leq \alpha \leq Nl ,1\leq \alpha^{\prime}\leq Nl)} $ be an $ Nl\times Nl $ matrix. Its $ (\alpha ,\alpha^{\prime}) $-entry(for $ 1\leq \alpha \leq Nl ,1\leq \alpha^{\prime}\leq Nl $) is $ \tilde{H}_{N}(z,E)(\alpha ,\alpha^{\prime}) $.

Consider the holomorphic function $ det(\tilde{H}_{N}(z,E)) $ on  $ \mathcal{A}_{\rho} $. By \eqref{a1},\eqref{a2} and Hadamard's inequality,  there exists a constant $ C_{6}>0 $,
 \begin{equation}\label{key}
	|det(\tilde{H}_{N}(z,E))|\leq (C_{6}(\lambda+|E|))^{Nl}.
\end{equation} Thus $u(z)=\dfrac{1}{Nl}\log |det(\tilde{H}_{N}(z,E))|$ is  subharmonic and satisfies  \begin{equation}\label{d6}
0\leq u(z)\leq log(C_{6}(\lambda+|E|))
\end{equation}  on  $ \mathcal{A}_{\rho} $.

Let  		
	\begin{equation}\label{b4}
		D_{N}(z,E):=\frac{1}{\sqrt{1+E^{2}}}diag\{[{F}_{j}(z)-\frac{E}{\lambda}I]M_{j}(z), 1\leq j\leq N\}.
	\end{equation}	
By \eqref{b4}, let   
	\begin{equation}\label{b5}
			B_{N}(z):=\tilde{H}_{N}(z,E)-\lambda D_{N}(z,E).
	\end{equation}
Denote by $ B_{N}(z)(i,j),(i,j=1,...,N) $ the $ (i,j) $-block of $ N\times N $ block matrix $ B_{N}(z) $. By  	\eqref{b1}, \eqref{b2} , \eqref{b3} , \eqref{b4} and  \eqref{b5},		
	\begin{equation}\label{c1}		B_{N}(z)(i,j)=\left\{
		\begin{aligned}
			\frac{1}{\sqrt{1+E^{2}}}R_{j}(z),&&  i=j,  \\
			\frac{1}{\sqrt{1+E^{2}}}(-\tilde{W}_{i}^{T}(z)),&&  i=j+1,  \\
			\frac{1}{\sqrt{1+E^{2}}}(-\tilde{W}_{j}(z)),&& i=j-1, \\
			0,&& { otherwise},
		\end{aligned}\right.
	\end{equation}	
where the entries of  $ R_{j}(z), -\tilde{W}_{i}^{T}(z) $ and $ -\tilde{W}_{j}(z) $ are
the analytic continuation to   strip $ \mathcal{A}_{\rho} $ of the corresponding entries  of $ R_{j}(x), -\tilde{W}_{i}^{T}(x) $ and $ -\tilde{W}_{j}(x) $ .

By \eqref{b4},
\begin{equation}\label{c4}
	det[D_{N}(z,E)]=(\frac{1}{\sqrt{1+E^{2}}})^{Nl}\prod\limits_{j=1}^{N}det[ ({F}(z+j\omega)M(z+j\omega)-\frac{E}{\lambda}  M(z+j\omega))].
\end{equation}
By    \eqref{b5}, 	\begin{equation}\label{c5}
	\begin{aligned}
		det[\tilde{H}_{N}(z,E)]=det[\lambda D_{N}(z,E)+B_{N}(z)].
	\end{aligned}
\end{equation}
 
It follows from \eqref{a8}, \eqref{a9} and \eqref{a0} that the energy $ E $ in $ G_{N}(x,E) $ is not restricted to a compact interval of $ \mathbb{R} $. Then for constant $ C_{2} >0$, if 

(i) $ |E|\leq C_{2} \lambda $, by \eqref{a6}, 
\begin{equation}\label{b6}
	{	det}[\frac{1}{\sqrt{1+E^{2}}}(F(z)-\frac{E}{\lambda}I)M(z)]\not\equiv 0 
\end{equation}
as a function of $ z $ on the strip $ \mathcal{A}_{\rho} $.
Let $ \mu=\frac{E}{\lambda} $.
By \eqref{b6} and the anlyticity of $ {	det}[(F(z)-\mu I)M(z)] $, for all   $ 0<\delta<\rho $ and fixed $\mu$ , the number of zeros of $ {	det}[\frac{1}{\sqrt{1+E^{2}}}(F(z)-\mu I)M(z)] $ in the strip $\mathcal{B}_{\delta}= \{z:z=x+\sqrt{-1}y,x\in [0,1],\frac{\delta}{2}\leq |y|\leq 2\delta \} $  is finite. Then there is $ \varepsilon_{1}=\varepsilon_{1}(\delta)>0 $,

\begin{equation}\label{key}
	\inf\limits_{|\mu|\leq C_{2}}\sup\limits_{\frac{\delta}{2}\leq |y|\leq 2\delta} \inf\limits_{x\in[0,1]}{	det}[\frac{1}{\sqrt{1+E^{2}}}(F(z)-\mu I)M(z)]\geq \varepsilon_{1}, z=x+\sqrt{-1}y.
\end{equation}

(ii) $ |E|> C_{2}\lambda $, by \eqref{a6}, 
\begin{equation}\label{}
	{	det}[\frac{1}{\sqrt{1+E^{2}}}F(z)-\frac{E}{\sqrt{1+E^{2}}}\frac{1}{\lambda}M(z)]\not\equiv 0 
\end{equation}
as a function of $ z $ on the strip $ \mathcal{A}_{\rho} $. Let $ \mu=\frac{E}{\sqrt{1+E^{2}}} $, we have 
\begin{equation}\label{z1}
	{	det}[\sqrt{1-\mu^{2}}F(z)-\frac{\mu}{\lambda}M(z)]\not\equiv 0 
\end{equation}
as a function of $ z $ on the strip $ \mathcal{A}_{\rho} $.
By \eqref{z1} and the anlyticity of $ 	{	det}[\sqrt{1-\mu^{2}}F(z)-\frac{\mu}{\lambda}M(z)] $, for all   $ 0<\delta<\rho $ and fixed $\mu$, the number of zeros of $	{	det}[\sqrt{1-\mu^{2}}F(z)-\frac{\mu}{\lambda}M(z)] $ in the strip $\mathcal{B}_{\delta}= \{z:z=x+\sqrt{-1}y,x\in [0,1],\frac{\delta}{2}\leq |y|\leq 2\delta \} $  is finite. Then there is $ \varepsilon_{2}=\varepsilon_{2}(\delta)>0 $, 

\begin{equation}\label{key}
	\inf\limits_{|\mu|\leq 1}\sup\limits_{\frac{\delta}{2}\leq |y|\leq 2\delta} \inf\limits_{x\in[0,1]}{	det}[\sqrt{1-\mu^{2}}F(z)-\frac{\mu}{\lambda}M(z)]\geq \varepsilon_{2}, z=x+\sqrt{-1}y.
\end{equation} 
It follows from $ (i) $ and $ (ii) $ that there exists constants  $\frac{\delta}{2}\leq \delta_{0}\leq 2\delta   ,  \varepsilon_{0}=\varepsilon_{0}(\delta_{0})=min\{\varepsilon_{1}(\delta_{0}), \varepsilon_{2}(\delta_{0})\}>0$ such that for $ z\in \mathcal{C}_{\delta_{0}}:= \{z:z=x+\sqrt{-1}y,x\in [0,1], y=\delta_{0} \} $,
\begin{equation}\label{c3}
	|det[\frac{1}{\sqrt{1+E^{2}}} (F(z)M(z)-\frac{E}{\lambda}  M(z))]|\geq \varepsilon_{0}.
\end{equation}
By \eqref{c4} and \eqref{c3}, for $ z\in \mathcal{C}_{\delta_{0}}$,
\begin{equation}\label{c6}
	det[D_{N}(z)]=\prod\limits_{j=1}^{N}det[\frac{1}{\sqrt{1+E^{2}}} (F(z+j\omega)M(z+j\omega)-\frac{E}{\lambda}  M(z+j\omega))]\geq \varepsilon_{0}^{Nl}.
\end{equation}

By \eqref{c5}, 
	\begin{equation}
	det[\tilde{H}_{N}(z,E)]=det[\lambda D_{N}(z,E)+B_{N}(z)]
\end{equation}
\begin{equation}\label{c9}
	=det[\lambda D_{N}(z)(I+\lambda^{-1}D_{N}^{-1}(z,E)B_{N}(z))]\\
	=\lambda^{Nl}det[D_{N}(z,E)]det[I+\lambda^{-1}D_{N}^{-1}(z,E)B_{N}(z)].
\end{equation}
By \eqref{c6}, for $ z\in \mathcal{C}_{\delta_{0}}$ there exists a constant  $ C_{7}>0  $, 
\begin{equation}\label{c7}
	\vert	\vert D_{N}^{-1}(z,E) 	\vert	\vert=	\vert	\vert\frac{adj[D_{N}(z,E)]}{det[D_{N}(z,E)]}	\vert	\vert\leq \frac{C_{7}^{Nl}}{\varepsilon_{0}^{Nl}}=(C_{7}\varepsilon_{0}^{-1})^{Nl},
\end{equation}
where $ adj[D_{N}(z,E)] $ is the adjoint matrix of $ D_{N}(z,E) $. By \eqref{c7},  for $ z\in \mathcal{C}_{\delta_{0}}$ and  $\lambda$ large enough  there exists a constant  $ C_{8}>0  $ ,

	\begin{equation}\label{c8}
	\vert\vert\lambda^{-1}D_{N}^{-1}(z,E)B_{N}(z)\vert\vert\leq \lambda^{-1}\vert\vert D_{N}^{-1}(z,E)\vert\vert
	\vert\vert B_{N}(z)\vert\vert\leq \lambda^{-1}(C_{8}\varepsilon_{0}^{-1})^{Nl}\leq \frac{1}{2}.
\end{equation}

By \eqref{c8}, for $ z\in \mathcal{C}_{\delta_{0}}$ ,
\begin{equation}\label{d3}		
	det[I+\lambda^{-1}D_{N}^{-1}(z,E)B_{N}(z)]\geq (\frac{1}{2})^{Nl}.
\end{equation}
Combine \eqref{c6}, \eqref{c9} and \eqref{d3}, for $ z\in \mathcal{C}_{\delta_{0}}$ we have 
	\begin{equation}
	det[\tilde{H}_{N}(z,E)]\geq 
	\lambda^{Nl}\epsilon_{0}^{Nl} (\frac{1}{2})^{Nl}.
\end{equation}
Thus for $ z\in \mathcal{C}_{\delta_{0}}$,
	\begin{equation}\label{d5}
	u(z)=\frac{1}{Nl}log|det[\tilde{H}_{N}(z,E)]|\geq log\lambda +log \frac{\epsilon_{0}}{2}.
\end{equation}

By subharmonicity, there exists a constant  $  \beta\in(0,1) $  such that 		
\begin{equation}\label{d4}
	(1-\beta)\int_{z\in \mathbb{T}}^{}u(z)dz\geq \int_{z\in \mathcal{C}_{\delta_{0}}}^{}u(z)dz-\beta \int_{z\in \mathcal{C}_{\rho}}^{}u(z)dz,
\end{equation}
where $ \mathcal{C}_{\rho}:= \{z:z=x+\sqrt{-1}y,x\in [0,1], y=\rho \} $.
By \eqref{d4},  
\begin{equation}\label{d44}
	\int_{z\in \mathbb{T}}^{}u(z)dz\geq \frac{1}{(1-\beta)}(\int_{z\in \mathcal{C}_{\delta_{0}}}^{}u(z)dz-\beta \int_{z\in \mathcal{C}_{\rho}}^{}u(z)dz).
\end{equation}
By \eqref{d6} , \eqref{d5}  and \eqref{d44}, there exists a constant $ C_{9}>0 $ ,
	\begin{equation}\label{d7}
	\int_{z\in \mathbb{T}}^{}u(z)dz \geq \frac{1}{1-\beta}(log|\lambda|+log \frac{\epsilon_{0}}{2}-\beta (log\lambda+C_{9})).
\end{equation}
Thus by \eqref{d7}, there exists a constant $ C_{1}>0 $,
\begin{equation}\label{d9}
	\int_{\mathbb{T}} \dfrac{1}{Nl}\log |det(\tilde{H}_{N}(x,E))|dx \geq \log\lambda-C_{1}.
\end{equation}
\end{proof}

\section{Green's fucntion estimate and the proof of  Anderson localization}
\begin{lemma}
 	{  Let $u(x):=\dfrac{1}{Nl}\log |det(\tilde{H}_{N}(x,E))|$, for all  integers $ Q\geq C_{0}^{-2} $ , there exists   constants $  \sigma > 0, c_{10}>0,  $ and $ S>0 $,  }                                                        \begin{equation}\label{d8}
 		{ mes}\{x\in \mathbb{T}:|\frac{1}{Q}\sum_{j=0}^{Q-1}u(x+j\omega)-\int_{\mathbb{T}}u(x)dx|\ge SQ^{-\sigma}\}< e^{-c_{10}Q^{\sigma}}.
 	\end{equation}
\end{lemma}	

\begin{proof}
	See  \cite{ref10}.
\end{proof}

\begin{proposition}
	{Assume} that  $ Q $ large enough. Then for all $1\leq \alpha \leq Nl ,1\leq \alpha^{\prime}\leq Nl $ there 	exists some $ j $ with $ 1\le j\le Q $  and constant $ c_{11} $ such that 
	
	\begin{equation}\label{e6}
		G_{N}(x+j\omega,E)(\alpha ,\alpha^{\prime})< e^{-|p(\alpha)-p(\alpha^{\prime})|log(\lambda+|E|)+c_{11} Nl}
	\end{equation} 
	holds for  all $ x $ except  for a set $\Omega(E)\subset\mathbb{T}$ of  measure at most 
	$ e^{-c_{10}Q^{\sigma}}.	 $
\end{proposition}

\begin{proof}
	By \eqref{d9} and \eqref{d8}, 
	\begin{equation}
		\frac{1}{Q}\sum_{j=1}^{Q-1}u(x+j\omega)\geq \int_{\mathbb{T}}u(x)dx-SQ^{-\sigma}\geq\log\lambda-C_{1}-SQ^{-\sigma} 
	\end{equation}  except for a set of  measure at most $ e^{-C_{10}Q^{\sigma}} $.
Moreover, there exists some $ j $ with $ 1\le j\le Q $, such that
	\begin{equation}\label{e3}
		u(x+j\omega)\geq \int_{\mathbb{T}}u(x)dx-SQ^{-\sigma}\geq\log\lambda-C_{1}-SQ^{-\sigma} 
	\end{equation}
holds except for a set of  measure at most $ e^{-C_{10}Q^{\sigma}} $. Then by \eqref{a4}, \eqref{e2} and \eqref{e3},  for 
$1\leq \alpha \leq Nl ,1\leq \alpha^{\prime}\leq Nl $ , there exists some $ j $ with $ 1\le j\le Q $ , \begin{equation}\label{e5}
	|\tilde{H}^{-1}_{N}(x+j\omega,E)(\alpha ,\alpha^{\prime})|\leq e^{-|p(\alpha)-p(\alpha^{\prime})|log(\lambda+|E|)+Nl\tilde{C}}
\end{equation} holds except for a set of  measure at most $ e^{-C_{10}Q^{\sigma}} $. $ \tilde{C}=C+C_{1}+SQ^{-\sigma}+log(\frac{1}{\lambda}+\frac{1}{|E|}) $.

 Thus by \eqref{e4} and \eqref{e5}, for 
$1\leq \alpha \leq Nl ,1\leq \alpha^{\prime}\leq Nl $ there exists a constant $ c_{11} $ and   some $ j $ with $ 1\le j\le Q $ ,
	\begin{equation}\label{key}
		|G_{N}(x,E)(\alpha,\alpha^{\prime})|\leq e^{-|p(\alpha)-p(\alpha^{\prime})|log(\lambda+|E|)+c_{11}Nl }
	\end{equation}
holds except for a set of  measure at most $ e^{-C_{10}Q^{\sigma}} $.

\end{proof}

	Assume that an integer $ N_{0}>0 $ large enough. By \eqref{e6}, for 
	$1\leq \alpha \leq N_{0}l ,1\leq \alpha^{\prime}\leq N_{0}l $, 
	\begin{equation}\label{e7}
		|G_{N_{0}}(x,E)(\alpha,\alpha^{\prime})|\leq e^{-|p(\alpha)-p(\alpha^{\prime})|log(\lambda+|E|)+c_{11}N_{0}l }
	\end{equation}
holds except for a set $ \Omega(E) $ of  measure at most $ e^{-C_{10}Q^{\sigma}} $. 

For all $1\leq \alpha \leq N_{0}l ,  1\leq \alpha^{\prime}\leq N_{0}l $, write  $ \tilde{H}_{N_{0}}(x,E)(\alpha ,\alpha^{\prime})=\sum_{k\in \mathbb{Z}}^{}\hat{v}(k) e^{ikx}, |\hat{v}(k)|<e^{-\rho|k|}$. 
Substitute $ \tilde{H}_{N_{0}}(x,E)(\alpha ,\alpha^{\prime}) $  by  $ \sum_{|k|<C_{11}N_{0}}^{}\hat{v}(k) e^{ikx} $ for a constant $ C_{11}>0$.
By \eqref{a4}, substitute $ N_{0} $ for $ N $. Then  replace \eqref{e7} by the condition 	\begin{equation}\label{e8}
	\sum_{1\leq \alpha \leq N_{0}l ,1\leq \alpha^{\prime}\leq N_{0}l}^{}e^{2|p(\alpha)-p(\alpha^{\prime})|log(\lambda+|E|)}[\tilde{\mu}_{N_{0}}(x,E)_{(\alpha,\alpha^{\prime})}]^{2}\leq e^{2N_{0}l\tilde{C}}[{ det}(\tilde{H}_{N_{0}}(x,E))]^{2}.
\end{equation}

Thus \eqref{e8} is of the form 
\begin{equation}\label{a954}
	PO(cos\omega,sin\omega,cosx,sinx,E)\geq0,
\end{equation}
where $ PO $ is a polynomial of degree at most $ C(N_{0}l)^{2} $. Furthermore, truncate power series for "cos" and "sin" and replace by a polynomial 
	\begin{equation}\label{key}
		PO_{1}(\omega,x,E)\geq0
	\end{equation}
of degree at most $ (N_{0}l)^{3} $.
\begin{lemma}
	{ For fixed $ \omega\in { DC}_{A,C_{0}} $ and $ E $, $ \Omega(E) $ in 
	\eqref{e7} does only satisfy the measure estimate $ mes(\Omega(E))<e^{-C_{10}Q^{\sigma}} $ but also be assumed semialgebraic of degree
	at most $ (N_{0}l)^{3} $.}
	\end{lemma}	
\begin{proof}
	See \cite{ref4}.
\end{proof}	
Let $  [u,v]\subset\mathbb{Z}$. Let 
\begin{equation}\label{a00}
	H_{[u,v]}(x):=R_{[u,v]}H_{\lambda}(x)R_{[u,v]},
\end{equation}
where $ R_{[u,v]} =$ coordinate restriction to $[u,v]\subset\mathbb{Z}$, and the associated Green's function $ G_{[u,v]}(x,E)=[H_{[u,v]}(x)-E)]^{-1} $.	
\begin{lemma}
	{ Fix $ x_{0}\in \mathbb{T} $.  Consider the orbit $\{x_{0}+j\omega: |j|\leq N_{1}\}$ where $ N_{1}=N_{0}^{C_{12}} $large enough  for a constant $ C_{12}>0 $.  Then except for at most $ N_{1}^{1-\sigma_{1}} $ values of $ |j|<N_{1} $, for all $-N_{0}l\leq \alpha \leq N_{0}l $ and $   -N_{0}l\leq \alpha^{\prime}\leq N_{0}l $,
		\begin{equation}\label{e77}
		|G_{[-N_{0},N_{0}]}(x_{0}+j\omega,E)(\alpha,\alpha^{\prime})|\leq e^{-|p(\alpha)-p(\alpha^{\prime})|log(\lambda+|E|)+c_{11}N_{0}l },
	\end{equation}
where $ \sigma_{1} $ is a constant with $ 0<\sigma_{1}<1 $.}

\end{lemma}
\begin{proof}
	See \cite{ref4}.
\end{proof}	
Assume that $ \varphi=(\vec{\varphi}_{n})_{n\in \mathbb{Z}} $ satisfies 
\begin{equation}\label{key}
	\vert\vert \vec{\varphi}_{n}\vert\vert_{2}<C_{13}|n|, |n|\rightarrow \infty 
\end{equation}
for a constant $ C_{13}>0 $.  For fixed $ x_{0} $ in \eqref{e77},
\begin{equation}\label{f2}
	(H_{\lambda}(x_{0})-E)\varphi=0.
\end{equation}

 By \eqref{a00} and \eqref{f2}, 

\begin{equation}\label{f3}
	(R_{[u,v]}(H_{\lambda}(x_{0})-E)R_{[u,v]})\varphi=-(R_{[u,v]}(H(x_{0}))R_{Z\setminus([u,v])}) \varphi=\left\{
	\begin{aligned}
		\vec{\varphi}_{j}=W^{T}_{u}(x)\vec{\varphi}_{u-1},  j=u,  \\
		\vec{\varphi}_{j}=W_{v+1}(x)\vec{\varphi}_{v+1},  j=v,  \\
		\vec{\varphi}_{j}=\vec{0},{ j\in\mathbb{Z}\setminus\{u,v\}}.
	\end{aligned}\right.
\end{equation}
For 
$-ul\leq p_{1} \leq ul ,-ul\leq p_{2}\leq ul $, by $ G_{N}(x,E)(p_{1} ,p_{2}) $ and  $ G_{[u,v]}(x,E)(p_{1} ,p_{2}) $ denote the 	$ (p_{1} ,p_{2}) $-block of  $ N\times N $ block matrix $ G_{N}(x,E) $ and $ G_{[u,v]}(x,E) $ respectively.
Hence by \eqref{f3}, for $ n\in [u,v] $, 

\begin{equation}\label{key}
	\vert\vert \vec{\varphi}_{n}\vert\vert_{2}<\mathop{max}\limits_{p_{1}\in\{u,v\}}
	\vert\vert G_{[u,v]}(x_{0},E)(p_{1} ,p_{2})   \vert\vert_{\infty} (\mathop{max}\limits_{|k|<N_{1}+N_{0}}\vert\vert \vec{\varphi}_{k}\vert\vert_{2})
\end{equation}

\begin{equation}\label{key}
	<\mathop{max}\limits_{p_{1}\in\{u,v\}}
\vert\vert G_{[-N_{0},N_{0}]}(x_{0}+j\omega,E)(p_{1}-j,p_{2}-j)   \vert\vert_{\infty} N_{1}.
\end{equation}
By \eqref{e77}, we have 

\begin{equation}\label{f5}
	\vert\vert \vec{\varphi}_{n}\vert\vert_{2}<N_{1}e^{c_{14}N_{0}l}(e^{-log(\lambda+|E|)|u-n|}+e^{-log(\lambda+|E|)|v-n|})
\end{equation}
for a constant $ c_{14}>0 $.
Let $ [u,v]=[-N_{0}+j,N_{0}+j] $ and $ n=j $. By \eqref{f5}, we have $ |j-u|\geq N_{0}/2, |j-v|\geq N_{0}/2 $. 

By \eqref{e77}, there exists a constant $ c_{15}=c_{15}(\lambda,E,l)>0 $, 
\begin{equation}\label{f6}
	\vert\vert \vec{\varphi}_{j}\vert\vert_{2}<e^{-c_{15}N_{0}}
\end{equation}
holds for all $ |j|<N_{1} $ except $ N_{1}^{1-\sigma_{1}} $ many.

\begin{proposition}
{ Denote \begin{equation}\label{key}
		\mathcal{E}=\mathcal{E}_{\omega}=\mathop{\bigcup}_{|j|\le N_{1}}SpecH_{[-j,j]}(x_{0}),
	\end{equation} where  Spec$H_{[-j,j]}(x_{0}) $ is the spectrum of $ H_{[-j,j]}(x_{0}) $.
	Then if 	\begin{equation}
	x\notin \mathop{\bigcup}_{E^{\prime}\in\mathcal{E}_{\omega}}\Omega(E^{\prime}),
\end{equation}for all $-N_{0}l\leq \alpha \leq N_{0}l ,  -N_{0}l\leq \alpha^{\prime}\leq N_{0}l $
we have 

	\begin{equation}\label{e777}
	|G_{[-N_{0},N_{0}]}(x,E)(\alpha,\alpha^{\prime})|\leq e^{-|p(\alpha)-p(\alpha^{\prime})|log(\lambda+|E|)+c_{11}N_{0}l }.
\end{equation}

}
	
\end{proposition}
\begin{proof}Let $ j_{0} $ be an positive integer. By \eqref{f3}, 
	\begin{equation}\label{f333}
		(R_{[-j_{0}+1,j_{0}-1]}(H_{\lambda}(x_{0})-E)R_{[-j_{0}+1,j_{0}-1]}\varphi=\left\{
		\begin{aligned}
			\vec{\varphi}_{j}=W^{T}_{-j_{0}+1}(x)\vec{\varphi}_{-j_{0}},  j=-j_{0}+1,  \\
			\vec{\varphi}_{j}=W_{j_{0}}(x)\vec{\varphi}_{j_{0}},  j=j_{0}-1,  \\
			\vec{\varphi}_{j}=\vec{0},{ j\in\mathbb{Z}\setminus\{-j_{0}+1,j_{0}-1\}}.
		\end{aligned}\right.
	\end{equation}
Then  we have 
	\begin{equation}
\begin{aligned}
		1=\vert\vert\vec{\varphi}_{0}\vert\vert_{2}
	\leq  
	\vert\vert G_{[-j_{0}+1,j_{0}-1]}(x_{0},E)(0,j_{0}-1)\vert\vert_{\infty} \vert\vert\vec{\varphi}_{j_{0}}\vert\vert_{2}\\+
	\vert\vert G_{[-j_{0}+1,j_{0}-1]}(x_{0},E)(0,-j_{0}+1)\vert\vert_{\infty}  \vert\vert\vec{\varphi}_{-j_{0}}\vert\vert_{2}\\
\leq	\vert\vert G_{[-j_{0}+1,j_{0}-1]}(x_{0},E)\vert\vert_{\infty}(\vert\vert\vec{\varphi}_{j_{0}}\vert\vert_{2}+\vert\vert\vec{\varphi}_{-j_{0}}\vert\vert_{2}).
\end{aligned}
\end{equation}
If $  j_{0} $ and $ -j_{0} $ both satisfy \eqref{f6}, we have 

\begin{equation}\label{key}
		\vert\vert G_{[-j_{0},j_{0}]}(x_{0},E)\vert\vert_{\infty}>e^{-c_{15}N_{0}},
\end{equation}
or equivalently 
\begin{equation}\label{555}
	dist(E,SpecH_{[-j_{0},j_{0}]})<e^{-\frac{c_{0}}{2}N_{0}}.
\end{equation}By \eqref{f5} and \eqref{555}, we have \eqref{e777} holds.
\end{proof}

\begin{lemma}
	{ Let $ N_{2}=N_{0}^{c_{16}} $ with $ c_{16}>0 $ a sufficiently large constant. Then for all $\omega\in{ DC}_{A,C_{0}}$ except for a set of  measure at most $ N_{2}^{-1/13} $ we have 
	\begin{equation}\label{f8}
		x_{0}+n\omega\notin \mathop{\bigcup}_{E^{\prime}\in\mathcal{E}_{\omega}}\Omega(E^{\prime}) 
	\end{equation}
for all $ N_{2}^{1/2}<|n|<2N_{2} $.
\begin{proof}
	See \cite{ref4}.
\end{proof}

}
\end{lemma}
For each $ N_{2}^{1/2}<|n|<N_{2} $      and  for all $(-N_{0}+n)l\leq \alpha \leq (N_{0}+n)l ,  (-N_{0}+n)l\leq \alpha^{\prime}\leq (N_{0}+n)l $,  we have 
	\begin{equation}\label{eb777}
	|G_{[-N_{0}+n,N_{0}+n]}(x_{0},E)(\alpha,\alpha^{\prime})|\leq e^{-|p(\alpha)-p(\alpha^{\prime})|log(\lambda+|E|)+c_{11}N_{0}l }.
\end{equation}
Define the interval 
	\begin{equation}\label{f7}
		\Lambda=\mathop{\bigcup}_{N_{2}^{\frac{1}{2}}<n<2N_{2}}[-N_{0}+n,N_{0}+n]\supset [N_{2}^{\frac{1}{2}},2N_{2}].
\end{equation}
Apply the resolvent identity on \eqref{f7}. Then for all $ \alpha,\alpha^{\prime} $ with  $p(\alpha),p(\alpha^{\prime})\in \Lambda $, 
\begin{equation}\label{ea77}
	|G_{\Lambda}(x_{0},E)(\alpha,\alpha^{\prime})|\leq e^{-|p(\alpha)-p(\alpha^{\prime})|log(\lambda+|E|)},  |p(\alpha)-p(\alpha^{\prime})|>\frac{1}{10}N_{2},
\end{equation} where $ G_{\Lambda}(x_{0},E)=(R_{\Lambda}H_{\lambda}(x_{0})R_{\Lambda}-E)^{-1} $. Therefore by \eqref{f6} we have  
	
	\begin{equation}\label{key}
		\vert\vert \vec{\varphi}_{j}\vert\vert_{2}<e^{-c_{15}j}, \frac{1}{2}N_{2}<|j|<N_{2}.
	\end{equation}
For  integer $ j<0 $ , the process is the same.
	By\eqref{f8}, for fixed  $ N_{0}$, for all $\omega\in{ DC}_{A,C_{0}}$ except for a set $ \mathcal{R}_{(N_{0})} $ of  measure at most $ N_{0}^{-10} $,

		\begin{equation}\label{key}
		\vert\vert \vec{\varphi}_{j}\vert\vert_{2}<e^{-c_{15}j}, \frac{1}{2}N_{2}<|j|<N_{2}.
	\end{equation} 
	Let 
	
	\begin{equation}
		\mathcal{R}=\mathop{\bigcap}_{N}\mathop{\bigcup}_{N_{0}>N} \mathcal{R}_{(N_{0})}.
	\end{equation}
	Then    \begin{equation}\label{key}
		mes(\mathcal{R})\leq \mathop{lim}\limits_{N\rightarrow\infty}\sum_{N_{0}=N+1}^{+\infty}  mes  (\mathcal{R}_{(N_{0})})=0.
	\end{equation}

	Finally for fixed $ x_{0}\in \mathbb{T} $, Anderson localization holds for a.e $ \omega \in $ DC$ _{A,C_{0}} $.

	\clearpage

	\bibliographystyle{alpha}

\begin{thebibliography}{99} 
	\bibitem{ref15}Sinai Y G. Anderson localization for one-dimensional difference Schrödinger operator with quasiperiodic potential[J]. Journal of statistical physics, 1987, 46(5-6): 861-909.
	
	\bibitem{ref16}Fröhlich J, Spencer T, Wittwer P. Localization for a class of one-dimensional quasi-periodic Schrödinger operators[J]. 1990.
	
	 \bibitem{ref17}Eliasson L H. Discrete one-dimensional quasi-periodic Schrödinger operators with pure point spectrum[J]. 1997.
	
	\bibitem{ref18}Jitomirskaya S Y. Anderson localization for the almost Mathieu equation: a nonperturbative proof[J]. Communications in Mathematical Physics, 1994, 165: 49-57.
	
	\bibitem{ref19}Jitomirskaya S Y. Metal-insulator transition for the almost Mathieu operator[J]. Annals of Mathematics, 1999: 1159-1175.
	
		\bibitem{ref9}Bourgain J, Goldstein M. On nonperturbative localization with quasi-periodic potential[J]. Annals of Mathematics, 2000, 152(3): 835-879.
	
\bibitem{ref4}Bourgain J. Green's Function Estimates for Lattice Schrödinger Operators and Applications.(AM-158)[M]. Princeton University Press, 2004.	
	
	
	
	
	\bibitem{ref10}Bourgain J, Jitomirskaya S. Anderson localization for the band model[J]. Geometric aspects of functional analysis, 2000: 67-79.	
	
	\bibitem{ref2}Klein S. Anderson localization for one-frequency quasi-periodic block Jacobi operators[J]. Journal of Functional Analysis, 2017, 273(3): 1140-1164.
	
	\bibitem{ref12}
Jian W, Shi Y, Yuan X. Anderson localization for one-frequency quasi-periodic block operators with long-range interactions[J]. Journal of Mathematical Physics, 2019, 60(6): 063504.	
	
	
	
	
	
	
	 \bibitem{ref14}Anderson P W. Absence of diffusion in certain random lattices[J]. Physical review, 1958, 109(5): 1492.
	
	\bibitem{ref7}Lagendijk A, Van Tiggelen B, Wiersma D S. Fifty years of Anderson localization[J]. Phys. today, 2009, 62(8): 24-29.
	
	
	
		\bibitem{ref11}Bourgain J. Anderson Localization for Quasi-Periodic Lattice Schrödinger Operators on ${\mathbb {Z}}^{d} $, d Arbitrary[J]. GAFA Geometric And Functional Analysis, 2007, 17(3): 682-706.
	
	
	
	
		
	

	

	

	

	
	
	

	
	

	\bibitem{13}
	Silvius Klein,
	Anderson localization for the discrete one-dimensional quasi-periodic Schrödinger operator with potential defined by a Gevrey-class function,
	Journal of Functional Analysis,
	Volume 218, Issue 2,
	2005,
	Pages 255-292.
	
	\bibitem{ref5}Jian W, Shi J, Yuan X. Anderson localization for long-range operators with singular potentials[J]. Journal of Mathematical Physics, 2021, 62(2): 022703.
	
	\bibitem{ref1} Shnol' È È.On the behavior of the eigenfunctions of Schrodinger's equation[J] Matematicheskii Sbornik, 1957, 84(3): 273-286.
	
	\bibitem{ref3}Han R. Shnol’s theorem and the spectrum of long range operators[J]. Proceedings of the American Mathematical Society, 2019, 147(7): 2887-2897.
	
	
	
	
	\bibitem{ref6}Aigner M, Axler S. A course in enumeration[M]. Berlin: Springer, 2007.
	
	
	
	
\end{thebibliography}

% see references.bib for bibliography management
\end{document}